\documentclass[12pt, oneside]{amsart}   
\allowdisplaybreaks
\usepackage{hyperref}
\usepackage{geometry,comment}                		
\usepackage{graphicx}				
\usepackage{amssymb}

\newtheorem{theorem}{\bf Theorem}[section]
\newtheorem{proposition}[theorem]{\bf Proposition}
\newtheorem{lemma}[theorem]{\bf Lemma}
\newtheorem{corollary}[theorem]{\bf Corollary}

\newtheorem{remark}[theorem]{\bf Remark}

\newenvironment{proofof}[1]{\noindent{\it Proof of
#1.}}{\hfill$\square$\\\mbox{}}

\def\field{K}
\def\NN{\mathbb{N}}
\def\ZZ{\mathbb{Z}}
\def\QQ{\mathbb{Q}}
\def\RR{\mathbb{R}}
\def\CC{\mathbb{C}}
\def\FF{\mathbb{F}}
\def\sepbeta{\beta_{\mathrm{sep}}}

\title{Separating polynomial invariants over non-closed fields of finite abelian groups}
\author[M. Domokos]{M\'aty\'as Domokos}
\address{HUN-REN Alfr\'ed R\'enyi Institute of Mathematics,
Re\'altanoda utca 13-15, 1053 Budapest, Hungary,
ORCID iD: https://orcid.org/0000-0002-0189-8831}
\email{domokos.matyas@renyi.hu}
\subjclass[2020]{Primary 13A50; Secondary 13P15, 20C15, 94A12}
\keywords{polynomial invariants, separating sets, degree bounds, finite abelian groups, multireference alignment, regular representation}

\begin{document}
\maketitle

\begin{abstract}
It is proved that for any finite dimensional representation of a prime order group over the field of rational numbers, 
polynomial invariants of degree at most $3$ separate the orbits. 
A result providing an upper degree bound for separating invariants for representations of finite abelian groups  over algebraically closed base fields of non-modular 
characteristic  is generalized for the case of base  
fields that are not algebraically closed (like the fields of real or rational numbers). 
\end{abstract}

\section{Introduction} 

This work is largely motivated by  \cite{bandeire-blumsmith-kileel-niles-weed-perry-wein}, 
where the authors develop an approach to multireference alignment using invariant theory. 
Multireference alignment is a basic problem in applied mathematics, which arose in connection with structural biology, image recognition, and signal processing 
(see e.g. the introduction of \cite{bandeira-rigollet-weed}). 
It was shown in \cite{perry_etal}  that the sample complexity of multireference alignment is related to the orbit recovery problem for the cyclic group $\mathrm{C}_n$ 
acting by cyclic shifts on the space $\RR^n$.  
In particular, in  \cite[Theorem 4.1]{bandeire-blumsmith-kileel-niles-weed-perry-wein} a Galois theoretic proof is given for the fact that for the regular representation  
over $\RR$ of a finite abelian group,  
the field of rational invariants is generated by polynomial invariants of degree at most $3$ 
(and hence generic orbits can be separated by polynomial invariants of degree at most $3$).  This is in fact a special case of a theorem on the regular representation of an arbitrary 
finite group (see \cite[Remark 4.3]{bandeire-blumsmith-kileel-niles-weed-perry-wein},  
\cite{edidin-katz}; a generalization is given in \cite{blumsmith-derksen}),  
but here first we shall focus on the special case of cyclic groups of prime order $p$. 
It is proved in \cite[Proposition 3.13]{blumsmith-garcia-hidalgo-rodriguez} 
that for a faithful finite dimensional real representation of $\mathrm{C}_p$, 
polynomial invariants of degree smaller than $p$ are not sufficient to separate the orbits. 
In sharp contrast with that result,  our Theorem~\ref{thm:main} says that different orbits in $\QQ^p$ under the regular representation of $\mathrm{C}_p$ can always be separated by invariants of degree at most $3$. As a consequence, the same conclusion holds for any finite dimensional representation of $\mathrm{C}_p$ over the field $\QQ$ of rational numbers. As far as we know, this is the first example in the literature of a finite group whose separating Noether number over an infinite field $\FF$ 
is not equal to its separating Noether number over the algebraic closure 
of $\FF$ (see Section~\ref{sec:prel} for the definition of the separating Noether number). 
Also our Theorem~\ref{thm:main} gives the first example of a field $\FF$ and an infinite set of pairwise non-isomorphic groups 
such that the separating Noether numbers 
over $\FF$ of the groups in the given set stay below a uniform upper bound. 
Note that a similar scenario does not occur for the ordinary Noether number (defined in terms of generating invariants, as opposed to separating invariants), 
because \cite[Theorem 2.1]{cziszter-domokos:lower bound} (and \cite{richman} in the modular case) imply that for any field $\FF$, any infinite set of pairwise non-isomorphic finite groups contains an element 
with arbitrarily large Noether number. 
The study of separating invariants over the field of rational numbers may have 
interest from the point of view of applications, because the numerical data in applications are rational (have finite decimal representation), so the orbit recovery problem over $\QQ$ has relevance. 

We provide a theoretical explanation for the phenomenon revealed by Theorem~\ref{thm:main} by placing it into a more general context. 
In fact a second motivation for this paper is that while degree bounds for separating invariants of finite abelian groups over the 
complex field are relatively well understood (see \cite{domokos:abelian}, \cite{schefler_c_n^r}, \cite{schefler_rank2}, \cite{schefler-zhao-zhong}), 
only sporadic information is available about the exact degree bounds for real representations of these groups. It is desirable to fill in this gap in the 
literature, because several applications require to work over the real field (see for example the problems mentioned in the 
introduction of \cite{cahill-iverson-mixon}, or see \cite{cahill-iverson-mixon-packer}).   
Recall that every representation of a finite abelian group over an algebraically closed base field of non-modular characteristic 
decomposes as a direct sum of $1$-dimensional representations, and therefore the corresponding algebra of polynomial invariants 
is generated by monomials. Separating sets of monomials were characterized in \cite{domokos:abelian} (see \cite{domokos:BLMS} and 
\cite{dufresne-jeffries} for similar results on not necessarily finite diagonalizable group actions). 
Here in Theorem~\ref{thm:separating monomials} we work out the generalization of this result appropriate for representations of finite abelian groups over non-algebraically closed base fields, where we can not diagonalize a representation. Theorem~\ref{thm:separating monomials}  clarifies the reason why the degree bound for separating invariants 
is smaller in general, than the corresponding degree bound when we replace the base field by its algebraic closure. 

We begin in Section~\ref{sec:prel} by recalling some basic concepts and notation of invariant theory. 
In Section~\ref{sec:block monoid} we deduce from an observation on the block monoid Proposition~\ref{prop:deg 3 separating on torus}, 
which is a variant of 
\cite[Theorem 4.1]{bandeire-blumsmith-kileel-niles-weed-perry-wein} about generically 
separating invariants on the complex regular representation of a finite abelian group. 
This is then used to prove Theorem~\ref{thm:main} on the regular representation of $\mathrm{C}_p$ over the field $\QQ$ of rational numbers. 
We conclude Section~\ref{sec:rational} by Proposition~\ref{prop:C4} and Proposition~\ref{prop:S3}, showing that 
Theorem~\ref{thm:main} does not generalize for cyclic groups of composite order, or for non-abelian groups. 
In Section~\ref{sec:non-closed} we first recall in Proposition~\ref{prop:irrep abelian} the classification of irreducible representations of a finite abelian group $G$ 
over a not necessarily algebraically closed base field of non-modular characteristic. We state and prove Theorem~\ref{thm:separating monomials} giving a sufficient 
condition for a set of monomials (defined over the algebraic closure of the base field $\FF$) to separate the orbits in the regular representation of $G$ over $\FF$, 
and deduce in Corollary~\ref{cor:sepnoether} an improved  bound on the separating Noether number of $G$ over $\FF$.  
This improvement is illustrated by a concrete example in Proposition~\ref{prop:C3xC3} and Section~\ref{sec:an example}. 

\section{Preliminaries in invariant theory} \label{sec:prel} 

Let us recall some standard notation of invariant theory. 
Throughout this paper $G$ will be a finite group, and $\FF$ a field whose characteristic does not divide the order of $G$. 
Let $V$ be a finite dimensional vector space over $\FF$, endowed with a representation of $G$ on $V$; that is, we are given a 
homomorphism $\rho$ from $G$ into the group $\mathrm{GL}(V)$ of invertible linear transformations of $V$. 
Sometimes we shall refer to a representation as the pair $(V,\rho)$, or we say that $V$ is an $\FF G$-module. 
Write $\FF[V]$ for the coordinate ring of $V$. This is a polynomial algebra in $\dim_\FF(V)$ commuting variables 
endowed with the 
standard grading (where the variables have degree $1$). 
 The variables are thought of as the coordinate functions on $V$ with respect to a chosen basis, 
and the elements of $\FF[V]$ represent $\FF$-valued polynomial functions on $V$ (when $\FF$ is infinite, 
$\FF[V]$ can be identified with the 
algebra of $\FF$-valued  polynomial functions on $V$). 
We write $\deg(f)$ for the degree 
of a non-zero element $f$ of $\FF[V]$. 
The action of $G$ on $V$ induces an action of $G$ on $\FF[V]$ 
via graded $\FF$-algebra automorphisms 
(for $f\in \FF]V]$, $g\in G$, $v\in V$ we have $(g\cdot f)(v)=f(g^{-1}v)$),  
and we denote by $\FF[V]^G$ the subalgebra of polynomial 
$G$-invariants on $V$ (the polynomials fixed by each group element). 
This is a graded subalgebra of $\FF[V]$.  
For $v\in V$ we write $G\cdot v\subset V$ for the $G$-orbit of $v$. 
It is immediate from the definition that $G\cdot v=G\cdot w$ implies $f(v)=f(w)$ for all 
$f\in \FF[V]^G$. The converse also holds (see \cite[Theorem 3.12.1]{derksen-kemper}, 
\cite[Theorem 16]{kemper} or \cite[Lemma 3.1]{draisma-kemper-wehlau}): 
\begin{equation}\label{eq:separability by invariants} 
G\cdot v\neq G\cdot w\iff \exists f\in\FF[V]^G\colon f(v)\neq f(w).
\end{equation} 
The algebra $\FF[V]^G$ is generated by invariants of degree at most $|G|$ 
(see \cite{noether}, \cite{fleischmann}, \cite{fogarty}). 
We shall denote by $\sepbeta(G,V)$ the minimal positive integer $d$ such that 
for any $v,w\in V$ with $G\cdot v\neq G\cdot w$, there exists an $f\in\FF[V]^G$ of degree 
at most $d$ with $f(v)\neq f(w)$. 
Furthermore, we write $\sepbeta^\FF(G)$ for the maximum of $\sepbeta(G,V)$, where 
$V$ ranges over the finite dimensional $\FF G$-modules.  We call $\sepbeta^\FF(G)$ the 
\emph{separating Noether number of $G$ over $\FF$}; 
its study was initiated in  \cite{kemper}, \cite{kohls-kraft}.   
By the facts mentioned above, we have 
\begin{equation}\label{eq:sepbeta<beta} 
\sepbeta^\FF(G)\le \beta^\FF(G)\le |G|,
\end{equation} 
where $\beta^\FF(G)$ is the maximum of $\beta(G,V)$, as  
$V$ ranges over the finite dimensional $\FF G$-modules, and 
$\beta(G,V)$ is the maximal degree of an element in a minimal homogeneous generating system 
of $\FF[V]^G$. The study of the \emph{Noether number} 
$\beta^\FF(G)$ began in \cite{schmid}, and much of the presently known results on its exact value 
are summarized in \cite{cziszter-domokos-szollosi}. 

\section{An observation on the block monoid of a finite abelian group}\label{sec:block monoid} 

Let $G$ be a finite abelian group, written multiplicatively. 
Denote by $\mathcal{F}(G,\NN_0)$ the set of non-negative integer valued functions on $G$. This is a monoid, with pointwise addition of functions. 
It is a submonoid of the abelian group $\mathcal{F}(G,\ZZ)$ of integer valued functions on $G$, with pointwise addition of functions. 
Clearly $\mathcal{F}(G,\ZZ)$ is a free abelian group of rank $|G|$. 
For an element $s\in \mathcal{F}(G,\ZZ)$ we call 
\[|s|:=\sum_{g\in G}|s(g)|\]
the \emph{length of $s$}.  
The element $s$ of $\mathcal{F}(G,\NN_0)$ can be thought of as a sequence over $G$ of length $|s|$, 
where the order of the elements in the sequence is disregarded, and $s(g)$ is 
the number of occurrences of $g$ in the sequence $s$. 
The submonoid 
\[\mathcal{B}(G):=\{s\in \mathcal{F}(G,\NN_0)\mid \prod_{g\in G}g^{s(g)}=1_G\}\subseteq \mathcal{F}(G,\NN_0)\] 
is called the \emph{monoid of product-one sequences over $G$} (it is also called 
the \emph{block monoid of $G$}, influenced by the terminology of \cite{narkiewicz}). 
The abelian group $\mathcal{F}(G,\ZZ)$ is freely generated by $\delta_g$, $g\in G$, where 
\begin{equation}\label{eq:delta} \delta_g(x)=\begin{cases} 1 \text{ if }x=g;\\
0 \text{ if }x\neq g.\end{cases}\end{equation} 

\begin{lemma}\label{lemma:length 3 generators} 
The block monoid $\mathcal{B}(G)$ is contained in the $\ZZ$-submodule of $\mathcal{F}(G,\ZZ)$ 
generated by 
\begin{equation}\label{eq:S} 
S:=\{\delta_{1_G},\quad \delta_g+\delta_{g^{-1}}, \quad \delta_g+\delta_h+\delta_{(gh)^{-1}}\mid g,h\in G\setminus \{1_G\}\}.\end{equation} 
\end{lemma} 

\begin{proof}
Denote by $A$ the $\ZZ$-submodule in $\mathcal{F}(G,\ZZ)$ given as 
$A:=\{s\in \mathcal{F}(G,\ZZ)\mid \prod_{g\in G}g^{s(g)}=1_G\}$. 
Clearly $A$ contains $\mathcal{B}(G)$, therefore it is sufficient to show that 
$A$ is contained in $\langle S\rangle_\ZZ$, the $\ZZ$-submodule of $\mathcal{F}(G,\ZZ)$ generated by $S$. Take any $s\in A$. We prove by induction on $|s|$ that  $s$ belongs to 
$\langle S\rangle_\ZZ$. If $|s|=0$, then $s$ is the zero element of $\mathcal{F}(G,\ZZ)$, so it belongs to all $\ZZ$-submodules of $\mathcal{F}(G,\ZZ)$. Assume next that $|s|>0$, and all $s'\in A$ 
with $|s'|<|s|$ belong to $\langle S\rangle_\ZZ$. 
If $|s|=1$ then necessarily $s=\delta_{1_G}$ or 
$-\delta_{1_G}$, so $s$ belongs to $\langle \delta_{1_G}\rangle_\ZZ\le \langle S\rangle_\ZZ$. Moreover, if $s(1_G)\neq 0$, then replacing $s$ by $-s$ if necessary, 
we may assume that $s(1_G)>0$, hence for $s':=s-\delta_{1_G}\in A$ we have $|s'|=|s|-1$. By the induction 
hypothesis $s'\in \langle S\rangle_\ZZ$, implying in turn that 
$s=s'+\delta_{1_G}\in \langle S\rangle_\ZZ$. 
From now on we assume 
that $s(1_G)=0$. If $|s|=2$, then $s\in A$ implies $s=\pm(\delta_g+\delta_{g^{-1}})$, hence 
$s\in \langle S\rangle_\ZZ$ by definition of $S$. 
Assume finally that $|s|\ge 3$. Then $\sum_{s(g)>0}s(g)\ge 2$ or 
$\sum_{s(g)<0}s(g)\le -2$. Replacing $s$ by $-s$ if necessary we may assume 
that the first option holds, so there exist elements $g,h\in G\setminus \{1_G\}$ such that 
for $s':=s-(\delta_g+\delta_h)$ we have $|s'|=|s|-2$. 
Set $s'':=s-(\delta_g+\delta_h+\delta_{(gh)^{-1}})$. Then $s''\in A$ and $|s''|\le |s'|+1=|s|-1$, 
hence $s''\in \langle S\rangle_\ZZ$ by the induction hypothesis, implying in turn that 
$s=s''+(\delta_g+\delta_h+\delta_{(gh)^{-1}})\in \langle S\rangle_\ZZ$. 
\end{proof} 

Let $\field$ be  an extension of $\FF$ containing an element of multiplicative order $\mathrm{exp}(G)$, 
where $\mathrm{exp}(G)$ is the exponent of $G$ (the least common multiple of the orders of the elements of $G$). 
For example, we may take for $\field$ the algebraic closure of $\FF$. 
Denote by $\widehat G$ the group of characters of $G$. That is, 
$\widehat G$ is the set of group homomorphisms from $G$ into the multiplicative 
group of $\field$, endowed with pointwise multiplication. 
Then $\widehat G$ is isomorphic to $G$. 

Take a polynomial algebra $\field[x_\chi\mid \chi\in \widehat G]$ 
whose variables are labeled by the elements of $\widehat G$. 
The group $G$ acts on  $\field[x_\chi\mid \chi\in \widehat G]$ via $\field$-algebra automorphisms, such that for each $\chi\in \widehat G$ and $g\in G$ we have 
\begin{equation*}\label{eq:action on x_chi} 
g\cdot x_\chi=\chi(g)^{-1}x.\end{equation*} 
It follows that  each monomial spans a $G$-invariant subspace in the polynomial ring, 
and for $b\in \mathcal{F}(\widehat G,\NN_0)$, the monomial  
$x^b:=\prod_{\chi\in\widehat G}x_\chi^{b(\chi)}$ is $G$-invariant if and only if $b\in \mathcal{B}(\widehat G)$. 
Consequently, we have the well known equality 
\begin{equation}\label{eq:B(G) span} 
\field[x_\chi\mid \chi\in\widehat G]^G=\mathrm{Span}_\field\{x^b\mid b\in \mathcal{B}(\widehat G)\}.
\end{equation} 
Consider the monomials associated to  $S$ (defined in \eqref{eq:S}): 
\begin{equation}\label{eq:T} T:=\{x^s\mid s\in S\}= \{x_{1_{\widehat G}},\quad x_\chi x_{\chi^{-1}}, \quad 
x_\chi x_\rho x_{(\chi\rho)^{-1}}\mid \chi, \rho\in \widehat G\setminus \{1_{\widehat G}\}\}.
\end{equation}  
Obviously, the elements of $T$ are $G$-invariant. 

The elements of $\field[x_\chi\mid \chi\in \widehat G]$ represent $\field$-valued polynomial functions on the $|\widehat G|$-dimensional $\field$-vector space 
$V$ 
with basis $\{e_\chi\mid \chi\in \widehat G\}$, namely for 
$v=\sum_{\chi\in \widehat G}v_\chi e_\chi$ and $x_\rho$ we have  
$x_\rho(v)=v_\rho$. 
The action of $G$ on  $\field[x_\chi\mid \chi\in \widehat G]$ is in fact induced by the following 
representation of $G$ on $V$: 
\[\text{for }g\in G,\quad v=\sum_{\chi\in \widehat G}v_\chi e_\chi\in V
\quad\text{ we have }
g\cdot v=\sum_{\chi\in\widehat G}\chi(g)v_\chi e_\chi.\]  

\begin{proposition} \label{prop:deg 3 separating on torus} 
Let $v,w$ be vectors in $V$ such that 
$v_\chi\neq 0$ and $w_\chi\neq 0$ for all $\chi\in\widehat G\setminus \{1_{\widehat G}\}$. 
Then the $G$-orbit of $v$ coincides with the $G$-orbit of $w$ if and only if 
$f(v)=f(w)$ for all $f\in T$ (defined in \eqref{eq:T}). 
\end{proposition}

\begin{proof} 
 By Lemma~\ref{lemma:length 3 generators}, for any $b\in  \mathcal{B}(\widehat G)$, 
 there exist $t_1,\dots,t_k\in T$, $u_1,\dots,u_m\in T\setminus \{x_{1_{\widehat G}}\}$ such that  
 $x^b=t_1\cdots t_ku_1^{-1}\cdots u_m^{-1}$ (this  equality takes place in the 
 Laurent polynomial ring  $\field[x_\chi^{\pm 1}\mid \chi\in \widehat G]$). 
The variable $x_{1_{\widehat G}}$ does not appear in the elements of $T\setminus\{x_{1_{\widehat G}}\}$, 
therefore none of $u_1,\dots,u_m$ involves the variable  $x_{1_{\widehat G}}$, 
so $u_j(z)\neq 0$ for $z\in \{v,w\}$ and $j=1,\dots,m$.  
 Consequently, for $z\in \{v,w\}$ we have 
 $x^b(z)=t_1(z)\cdots t_k(z)u_1(z)^{-1}\cdots u_m(z)^{-1}$. 
 It follows that $f(v)=f(w)$ for all $f\in T$ if and only if $x^b(v)=x^b(w)$ for all 
 $b\in \mathcal{B}(\widehat G)$. 
Thus by \eqref{eq:B(G) span}, $f(v)=f(w)$ for all $f\in T$ if and only if $h(v)=h(w)$ for all 
$h\in \field[x_\chi\mid \chi\in \widehat G]^G$. 
Since different orbits of a finite group can be separated by polynomial invariants 
(see \eqref{eq:separability by invariants}), 
we conclude that $f(v)=f(w)$ for all $f\in T$ if and only 
$G\cdot v=G\cdot w$. 
 \end{proof} 

\section{Prime order cyclic groups over the rationals} \label{sec:rational}

Any finite group  $G$ acts on itself via left multiplication. This induces a linear action on the 
$\FF$-vector space $\mathcal{F}(G,\FF)$ of $\FF$-valued functions on $G$: 
for $f\in \mathcal{F}(G,\FF)$, $g,h\in G$ we have $(g\cdot f)(h)=f(g^{-1}h)$. 
This representation 
is called the \emph{regular representation of $G$ over $\FF$}. 
More concretely, the elements $\delta_g$ ($g\in G$) defined in \eqref{eq:delta} 
form a $\FF$-vector space basis in $\mathcal{F}(G,\FF)$, and for $h\in G$ we have 
$h\cdot \delta_g=\delta_{hg}$.  

When the characteristic of $\FF$ does not divide $|G|$, all irreducible representations 
of $G$ over $\FF$ occur as a direct summand in $\mathcal{F}(G,\FF)$.  Moreover, when 
$\FF$ is also algebraically closed, the multiplicity 
of each irreducible representation   of $G$ over 
$\FF$ equals the dimension of the representation. 

\begin{theorem}\label{thm:main}  
Consider the regular representation on $V:=\mathcal{F}(G,\QQ)$ of the cyclic group $G:=\mathrm{C}_p$  of  prime order $p$
over the field $\QQ$ of rational numbers. 
Then for any $v,w\in V$ with $G\cdot v\neq G\cdot w$ there exists an 
$f\in \QQ[V]^G$ with $\deg(f)\le 3$ and $f(v)\neq f(w)$. 
\end{theorem} 

\begin{proof}
Let $\omega$ be a primitive $p$th root of $1$, and $K:=\QQ(\omega)$ the $p$th cyclotomic extension of $\QQ$. 
For a character $\chi\in \widehat G$ set 
\[e_\chi:=\frac 1p\sum_{g\in G}\chi(g)^{-1}\delta_g, \] 
where  $\{\delta_g\mid g\in G\}$ is the basis of $\mathcal{F}(G,\QQ)$ defined in \eqref{eq:delta}.   
Then $\{e_\chi\mid\chi\in \widehat G\}$ is a $\field$-vector space basis of $V_\field:=\field\otimes_\QQ V$, 
satisfying 
\[g\cdot e_\chi=\chi(g)e_\chi \quad \text{ for all }g\in G, \quad \chi\in \widehat G.\] 
Note that $\chi(g)^{-1}$ is the complex conjugate 
of $\chi(g)$, and therefore the second orthogonality relation of characters 
(cf. \cite[Chapter 2, Proposition 7]{serre}) implies that 
\begin{equation}\label{eq:delta in basis e} 
\delta_g=\sum_{\chi\in\widehat G}\chi(g)e_{\chi} \quad \text{ for }g\in G.  
\end{equation}  
Now assume that for $v,w\in V$, $f(v)=f(w)$ for all $f\in \QQ[V]^G$ with 
$\deg(f)\le 3$. We claim that $G\cdot v=G\cdot w$. 
Since the $G$-action on $\field\otimes_\QQ V$ was obtained 
from the representation of $G$ on $V$ by extending scalars, each homogeneous component 
of $\field[\field\otimes_\QQ V]^G$ is spanned as a $\field$-vector space 
by the corresponding homogeneous component of $\QQ[V]^G$. 
Therefore we have that $h(v)=h(w)$ for all $h\in \field[\field\otimes_\QQ V]^G$ with 
$\deg(h)\le 3$. In particular, 
\begin{equation}\label{eq:h(v)=h(w)} 
h(v)=h(w) \quad\text{ for all }\quad h\in T, \end{equation} 
where $T$ is defined in 
\eqref{eq:T}, and $\field[\field\otimes_\QQ V]=\field[x_\chi\mid\chi\in \widehat G]$, 
and $\{x_\chi\mid \chi\in \widehat G\}$ is the basis dual to the basis 
$\{e_\chi\mid \chi\in \widehat G\}$ in $\field\otimes_\QQ V$.   

By \eqref{eq:delta in basis e}, the   $\QQ$-vector subspace $V=\{\sum_{g\in G}v_g\delta_g\mid v_g\in\QQ\}$ of 
$\field\otimes_\QQ V$ is 
\[V=\{\sum_{\chi\in\widehat G}(\sum_{g\in G}\chi(g)v_g)e_\chi\mid v_g\in\QQ\}.\] 
Note that for $\chi\in \widehat G\setminus\{1_{\widehat G}\}$ we have 
\[\{\chi(g)\mid g\in G\}=\{1,\omega,\omega^2,\dots,\omega^{p-1}\}.\] 
The minimal polynomial of $\omega$ over $\QQ$ is $1+x+x^2+\cdots+x^{p-1}$.  
Therefore $\sum_{g\in G}\chi(g)z_g=0$ for some 
$z=\sum_{g\in G}z_g\delta_g\in V$ and $\chi\in \widehat G\setminus\{1_{\widehat G}\}$ if and only if $z=c\sum_{g\in G}\delta_g$ for some $c\in \QQ$, 
or in other words, $z_g=c\in \QQ$ for all $g\in G$. Moreover, in this case $z=cpe_{1_{\widehat G}}$.  
So 
\[V\subset \QQ e_{1_{\widehat G}}\bigcup \{\sum_{\chi\in\widehat G}c_\chi e_\chi
\mid c_\chi\in \field,\ c_\chi\neq 0\text{ for }\chi\neq 1_{\widehat G}\}.\]

If both $v$ and $w$ belong to  $\QQ\sum_{g\in G}\delta_g$, then 
$v=c_1\sum_{g\in G}\delta_g=c_1pe_{1_{\widehat G}}$, $w=c_2\sum_{g\in G}\delta_g=c_2pe_{1_{\widehat G}}$, 
hence by \eqref{eq:h(v)=h(w)} we have $c_1p=x_{1_{\widehat G}}(v)=x_{1_{\widehat G}}(w)=c_2p$,  
implying  that $c_1=c_2$, and therefore $v=w$. 
 
If at least one of $v,w$, say $v$ does not belong to  $\QQ\sum_{g\in G}\delta_g$, then take any $\chi\notin\widehat G\setminus \{1_{\widehat G}\}$. 
We have that $(x_\chi x_{\chi^{-1}})(v) \neq 0$, whence  by \eqref{eq:h(v)=h(w)} we get 
$(x_\chi x_{\chi^{-1}})(w)\neq 0$, implying in turn that 
$w$ does not belong to  $\QQ\sum_{g\in G}\delta_g$. 

So it remains to deal with the case when  
none of $v$ and $w$ belong to $\QQ\sum_{g\in G}\delta_g$. Then their coordinates 
with respect to the basis $\{e_\chi\mid \chi\in \widehat G\}$ are all non-zero, except possibly the coordinate 
corresponding to $e_{1_{\widehat G}}$. By \eqref{eq:h(v)=h(w)} and 
Proposition~\ref{prop:deg 3 separating on torus} we conclude that $G\cdot v=G\cdot w$. 
\end{proof}

\begin{corollary}\label{cor:all rational reps} 
For the cyclic group $\mathrm{C}_p$ of prime order $p$ we have the equality 
\[\sepbeta^\QQ(\mathrm{C}_p)=\begin{cases} 3 &\quad \text{ for  }p\ge3; \\
2 &\quad \text{ for }p=2.\end{cases}\]  
\end{corollary} 

\begin{proof} 
For $p=2$ we have  $\sepbeta^\QQ(\mathrm{C}_2)\le 2$ by \eqref{eq:sepbeta<beta}, and  
the regular representation of $\mathrm{C}_2$ shows that we have equality here. 
For $p>2$ we know that $\sepbeta(\mathrm{C}_p,\mathcal{F}(\mathrm{C}_p,\QQ))\le 3$ 
by Theorem~\ref{thm:main}. Moreover, any finite dimensional representation 
of $\mathrm{C}_p$ over $\QQ$ is a direct summand of 
$\mathcal{F}(\mathrm{C}_p,\QQ)^{\oplus m}$, the direct sum of $m$ copies of the 
regular representation for some $m$, hence 
$\sepbeta^\QQ(\mathrm{C}_p)\le 
\max_{m\in \NN}\{\sepbeta(\mathrm{C}_p,\mathcal{F}(\mathrm{C}_p,\QQ)^{\oplus m})\}$. 
The latter number equals $\sepbeta(\mathrm{C}_p,\mathcal{F}(\mathrm{C}_p,\QQ)^{\oplus 2})$ 
by \cite[Proposition 4.3 and 4.4]{domokos:typical}. Moreover, 
$\sepbeta(\mathrm{C}_p,\mathcal{F}(\mathrm{C}_p,\QQ)^{\oplus 2})=\sepbeta(\mathrm{C}_p,\mathcal{F}(\mathrm{C}_p,\QQ))$  by \cite[Theorem 3.4]{draisma-kemper-wehlau}. 
So we proved the inequality $\sepbeta^\QQ(\mathrm{C}_p)\le 3$. 

To prove the reverse inequality for $p\ge 3$, consider the regular representation 
$\mathcal{F}(\mathrm{C}_p,\QQ)$. Denote by $g$ a generator of $\mathrm{C_p}$. 
Consider the element $v:=\delta_{1_G}-\delta_{g}$ and its negative $-v$. 
Denote by $\{x_h\mid h\in G\}$ the basis dual to the basis 
$\{\delta_h\mid h\in G\}$ in $\mathcal{F}(\mathrm{C}_p,\QQ)$. 
Then $\QQ[V]$ is the polynomial algebra $\QQ[x_h\mid h\in G]$, and 
$g\cdot x_h=x_{gh}$. 
Since $G$ acts on $\QQ[V]$ by permuting the variables, 
the homogeneous components of 
$\QQ[V]^{\mathrm{C}_p}$ are spanned by orbit sums of monomials. 
The only degree $1$ orbit sum $\sum_{h\in \mathrm{C}_p}x_h$ 
vanishes on both $v$ and $w$. 
The orbit sum of $x_{1_G}^2$ 
takes value 2 on both $v$ and $w$, whereas the orbit sum of 
$x_{1_G}x_g$ takes value $-1$ on both $v$ and $w$. All the other quadratic orbit sums vanish on both $v$ and $-v$. Thus $v$ and $-v$ can not be separated by invariants of degree at most $2$. However, they have different $\mathrm{C}_p$-orbits, 
because the orbit sum of $x_{1_G}^2x_g$ takes value $-1$ on $v$ and value $1$ on $-v$, 
when $p>2$. 
\end{proof} 

\begin{remark}
For comparison, we mention the well known  equality $\sepbeta^\CC(\mathrm{C}_p)=p$. 
Moreover, it is proved in \cite[Proposition 3.13]{blumsmith-garcia-hidalgo-rodriguez} that 
even over the field $\RR$ of real numbers we have 
$\sepbeta^\RR(\mathrm{C}_p)=p$. 
 \end{remark} 
 
Theorem~\ref{thm:main} does not extend to all cyclic groups whose order is not prime, 
as Proposition~\ref{prop:C4} below shows: 

\begin{proposition}\label{prop:C4} 
The points $v:=[3,4,-3.-4]^T\in \QQ^4$ and 
$w:=[5,0,-5,0]^T\in \QQ^4$ have different $\mathrm{C_4}$-orbit, 
where the cyclic group $\mathrm{C}_4$ acts on $\QQ^4$ by cyclically permuting the coordinates of vectors, and all polynomial $\mathrm{C}_4$-invariants on $\QQ^4$ of degree at most $3$ take the same value on $v$ and $w$. 
In particular, we have 
\[\sepbeta^\QQ(\mathrm{C}_4)=4.\]  
\end{proposition} 

\begin{proof}
Denote by $x_1,x_2,x_3,x_4$ the coordinate functions on $\QQ^4$, so for a generator 
$g$ of $\mathrm{C}_4$ we have $g\cdot x_1=x_2$, $g\cdot x_2=x_3$, 
$g\cdot x_3=x_4$, $g\cdot x_4=x_1$. The algebra $\QQ[x_1,x_2,x_3,x_4]^{\mathrm{C}_4}$ of polynomial $\mathrm{C}_4$-invariants is generated by orbit sums of monomials of 
degree at most $4$. Therefore it is easy to see that $\QQ[x_1,x_2,x_3,x_4]^{\mathrm{C}_4}$ is minimally generated by 
$f_1:=x_1+x_2+x_3+x_4$, 
$f_2:=x_1^2+x_2^2+x_3^2+x_4^2$, 
$f_3:=x_1x_2+x_2x_3+x_3x_4+x_4x_1$, 
$f_4:=x_1^3+x_2^3+x_3^3+x_4^3$, 
$f_5=x_1x_2^2+x_2x_3^2+x_3x_4^2+x_4x_1^2$, 
$f_6:=x_1^4 + x_2^4 + x_3^4 + x_4^4$, 
$f_7:=x_1x_2^3 + x_2x_3^3 + x_3x_4^3 + x_4x_1^3$. 
Now we have 
\begin{align*}
[f_1(v),f_2(v),f_3(v),f_4(v),f_5(v),f_6(v),f_7(v)]&=
[0, 50, 0, 0, 0, 674, 168]\\ 
[f_1(w),f_2(w),f_3(w),f_4(w),f_5(w),f_6(w),f_7(w)]&=[0, 50, 0, 0, 0, 1250, 0], 
\end{align*} 
showing that only the degree $4$ generators differ on $v$ and $w$.  
Thus we have $\sepbeta^\QQ(\mathrm{C}_4)\ge 4$. 
The reverse inequality holds by \eqref{eq:sepbeta<beta}. 
\end{proof}  

\begin{proposition}\label{prop:S3} 
Invariants of degree at most $3$ are not sufficient to separate all the orbits in 
the regular representation of $\mathrm{S}_3$ over the field of rational numbers. 
Moreover, we have 
\[\sepbeta^\QQ(\mathrm{S}_3)=4.\] 
\end{proposition} 

\begin{proof} 
The regular representation of $\mathrm{S}_3$ over $\QQ$ contains as a direct summand 
the sum of the standard $3$-dimensional permutation representation of  $\mathrm{S}_3$ 
and the sign representation of  $\mathrm{S}_3$. 
Therefore it is sufficient to show that the invariants of degree  at most $3$ do not separate all the orbits in this $4$-dimensional representation. 
The corresponding coordinate ring is $\QQ[x_1,x_2,x_3,y]$, a permutation $g\in \mathrm{S}_3$ maps $x_i$ to $x_{g(i)}$, and $g\cdot y=\mathrm{sign}(g)y$. 
It is well known and easy to see that 
$\QQ[x_1,x_2,x_3,y]^{\mathrm{S}_3}$ is minimally generated by 
$s_1:=x_1+x_2+x_3$, $x_2:=x_1x_2+x_2x_3+x_1x_3$, $x_1x_2x_3$, $y^2$, 
and $a:=y(x_2-x_1)(x_3-x_2)(x_3-x_1)$. 
Set $v:=[2,1,0,1]^T\in \QQ^4$, $w:= [2,1,0,-1]^T\in \QQ^4$. 
Since $v$ and $w$ have the same $x$-coordinates, $s_i(v)=s_i(w)$ for $i=1,2,3$. 
Moreover, since $y(v)=-y(w)$, we have $y^2(v)=y^2(w)$. So the invariants of degree at most $3$ agree 
on $v$ and $w$. 
However, $v$ and $w$ have different $G$-orbits, because $a(v)\neq a(w)$. This shows 
$\sepbeta^\QQ(\mathrm{S}_3)\ge 4$. 
The reverse inequality holds by \eqref{eq:sepbeta<beta}, because $\sepbeta^\QQ(\mathrm{S}_3)=4$ (this was shown first in 
\cite{schmid}; see \cite{cziszter-domokos-szollosi} for more references).    
\end{proof} 
 
\section{Representations of abelian groups over non-closed base fields}\label{sec:non-closed}

Throughout this section $G$ is an arbitrary finite abelian group. 
Let $\field=\FF(\omega)$ be a cyclotomic extension of $\FF$ by an element $\omega$ whose multiplicative order is 
$\mathrm{exp}(G)$, the least common multiple of the orders of elements of $G$. 
Note that the existence of $\field$ forces that $\mathrm{char}(\FF)$ does not divide $|G|$. 
Moreover, $\field$ is a Galois extension of $\FF$, denote by $\Gamma$ the automorphism group of the field extension $\field$ of   
$\FF$. 
Write $\widehat G$ for the group of characters of $G$ over $\field$; 
thus we have $\widehat G\cong G$ as groups. 
The Galois group $\Gamma$ acts naturally on $\widehat G$ and $\mathcal{F}(G,\field)$, namely 
for $\gamma\in \Gamma$ and $\chi\in \widehat G$,   
$\gamma\cdot \chi$ is the character of $G$ mapping $g\in G$ to $\gamma(\chi(g))$, and for 
$f\in\mathcal{F}(G,\field)$, $\gamma(f)$ is the function on $G$ mapping $g\in G$ to $\gamma(f(g))$. 
The $\FF$-subspace $\mathcal{F}(G,\FF)$ of $\mathcal{F}(G,\field)$ is the fixed point set of $\Gamma$ in 
$\mathcal{F}(G,\field)$. 
Denote by $\widehat G/\Gamma$ a chosen set of representatives of the $\Gamma$-orbits in $\widehat G$. 
The irreducible representations of $G$ over $\field$ are all $1$-dimensional.  For 
$\chi\in \widehat G$ denote by $(\field,\chi)$ the $1$-dimensional $\field$-vector space  $\field$ on which 
$G$ acts via the character $\chi$: for $g\in G$ and $\lambda\in \field$ we have 
$g\cdot \lambda=\chi(g)\lambda$. The irreducible representations of $G$ over $\FF$ are in a natural bijection with 
$\widehat G/\Gamma$, as follows. 

For each $\chi\in\widehat G/\Gamma$ we construct a representation of $G$ over $\FF$. 
The image $\chi(G)$ of $G$ under $\chi$ is a subgroup of the multiplicative group 
$\field^\times$ of $\field$, hence is cyclic. Its order $d_\chi$ is a divisor of $\mathrm{exp}(G)$. 
We 
choose an element $g_\chi\in G$ such that $\omega_\chi:=\chi(g_\chi)$ generates the subgroup $\chi(G)$ 
of $\field^\times$ (i.e. the multiplicative order of $\omega_\chi$ is $d_\chi$). 
The set of roots of 
the minimal polynomial $m^{\omega_\chi}_\FF$ of $\omega_\chi$ over $\FF$ is the 
$\Gamma$-orbit of $\omega_\chi$, whence $\ell_\chi:=\deg(m^{\omega_\chi}_\FF)=|\Gamma\cdot\omega_\chi|$. 
Let $A_\chi\in \FF^{\ell_\chi\times\ell_\chi}$ be the companion matrix of $m^{\omega_\chi}_\FF$, 
so the characteristic polynomial (and the minimal polynomial) of $A_\chi$ is $m^{\omega_\chi}_\FF$. 
Then $A_\chi^{d_\chi}$ is the $\ell_\chi\times\ell_\chi$ identity matrix, and the set of eigenvalues of 
$A_\chi$ is the $\Gamma$-orbit of $\omega_\chi$. 
Lift to $G$ the representation of $G/\ker(\chi)\cong \chi(G)$   given by 
$g_\chi^j\ker(\chi)\mapsto A_\chi^j$ for $j=0,1,\dots,d-1$. 
Denote by $\rho_\chi:G\to \mathrm{GL}(\FF^{\ell_\chi})$ the corresponding group homomorphism, 
so $(\FF^{\ell_\chi},\rho_\chi)$ is a representation of $G$ over $\FF$.  
Note that $\ker(\rho_\chi)=\ker(\chi)$ (because $\rho_\chi(g_\chi)=A_\chi$ has eigenvalues with multiplicative order $d_\chi=|G/\ker(\chi)|$). 
Moreover, we denote by $(\field^{\ell_\chi},\rho_\chi)$ the representation 
obtained from $(\FF^{\ell_\chi},\rho_\chi)$ by extending the scalars from $\FF$ to $\field$. 
The following statement must be well known, but we do not know a reference for it. 

\begin{proposition} \label{prop:irrep abelian} 
\begin{itemize}
\item[(i)] The representation $(\field^{\ell_\chi},\rho_\chi)$ is isomorphic to 
$\bigoplus_{\psi\in \Gamma\cdot\chi} (\field,\psi)$, where $\Gamma\cdot \chi$ is the 
$\Gamma$-orbit of $\chi$. 
\item[(ii)] The representation $(\FF^{\ell_\chi},\rho_\chi)$ is irreducible for all $\chi\in \widehat G/\Gamma$. 
\item[(iii)] The regular representation $(\mathcal{F}(G,\FF),\rho_{\mathrm{reg}})$ of $G$ over $\FF$ 
is isomorphic to the direct sum 
$\bigoplus_{\chi\in \widehat G/\Gamma}(\FF^{\ell_\chi},\rho_\chi)$,  
where $\widehat G/\Gamma$ stands for a complete list of representatives of the $\Gamma$-orbits in $\widehat G$. 
\item[(iv)]  $\{(\FF^{\ell_\chi},\rho_\chi)\mid \chi\in \widehat G/\Gamma\}$  is a complete irredundant list of 
representatives of the isomorphism classes of the irreducible representations of $G$ over $\FF$. 
\end{itemize}  
\end{proposition}

\begin{proof} 
(i): Since $A_\chi\in \field^{\ell_\chi\times\ell_\chi}$ has $\ell_\chi$ distinct eigenvalues in $\field$, and the set of eigenvalues 
is the $\Gamma$-orbit of $\omega_\chi$, the space $\field^{\ell_\chi}$ decomposes as the direct sum of $1$-dimensional eigenspaces 
$\field^{\ell_\chi}=\bigoplus_{\omega\in \Gamma\cdot \omega_\chi}\field v_\omega$, 
where $A_\chi v_\omega=\omega v_\omega$, implying that 
$g_\chi \cdot v_\omega =\rho_\chi(g_\chi)v_\omega=A_\chi v_\omega=\omega v_\omega$.  
Moreover, $\ker(\chi)=\ker(\rho_\chi)$, and since the coset of $g_{\chi}$ generates $G/\ker(\rho_\chi)$, 
we conclude that the subspace $\field v_\omega$ is $G$-invariant. 
Consequently, $(\field v_\omega,\rho_\chi\vert_{\field v_\omega})\cong (\field,\psi)$ for some 
$\psi\in \widehat G$. We claim that $\psi=\gamma\cdot \chi$, where $\gamma\in \Gamma$ satisfies 
$\omega=\gamma(\omega_\chi)$. 
Indeed, we have 
\[(\gamma\cdot\chi)(g_\chi)v_\omega=\gamma(\chi(g_\chi))v_\omega
=\gamma(\omega_\chi)v_\omega
=\omega v_\omega=A_\chi v_\omega=g_\chi\cdot v_\omega=\psi(g_\chi)v_\omega,\] 
showing that the characters $\gamma\cdot \chi$ and $\psi$ agree on $g_\chi$. 
Since  $\ker(\psi)\supseteq \ker(\rho_\chi)=\ker(\chi)$,  
$\ker(\gamma\cdot \chi)=\ker(\chi)$, and the coset of $g_\chi$ generates $G/\ker(\chi)$, 
we conclude that $\gamma\cdot \chi=\psi$.  
This shows that $(\field v_\omega,\rho_\chi\vert_{\field v_\omega})\cong (\field,\gamma\cdot\chi)$, 
and thus (i) holds.

(ii):  Let $V$ be a non-zero $G$-invariant subspace in $\FF^{\ell_\chi}$ (where $G$ acts on $\FF^{\ell_\chi}$ via $\rho_\chi$). 
The non-zero $A_\chi$-invariant subspace $\field\otimes_\FF V$  of $\field^{\ell_\chi}$ 
contains a non-zero $A_\chi$-eigenvector $v$. Its  
eigenvalue $\omega$ belongs to $\Gamma\cdot \omega_\chi$. 
The action of $\Gamma$ on $\field$ induces an action on 
$\field^{\ell_\chi}$ and on $\field^{\ell_\chi\times \ell_\chi}$ in the obvious way. The matrix $A_\chi$ is fixed by $\Gamma$, because its entries 
belong to $\FF$. For $\gamma\in\Gamma$ we have 
\begin{equation}\label{eq:gamma v}
A_\chi\gamma(v)=\gamma(A_\chi)\gamma(v)=
\gamma(A_\chi v)=
\gamma(\omega v)=\gamma(\omega)\gamma(v). 
\end{equation}   
Since $V$ is defined over $\FF$, 
$\Gamma$ maps $\field\otimes_\FF V$ into itself, so $\gamma(v)$ belongs to 
$\field\otimes_\FF V$, and it is an $A_\chi$-eigenvector with eigenvalue $\gamma(\omega)$. 
This holds for all $\gamma\in \Gamma$, so $\field\otimes_\FF V$ contains all the $A_\chi$-eigenvectors in 
$\field^{\ell_\chi}$.  
Therefore 
$\field\otimes_\FF V=\field^{\ell_\chi}$, and consequently, $V=\FF^{\ell_\chi}$.   
Thus $\FF^{\ell_\chi}$ has no proper non-zero $G$-invariant subspace, 
so the representation $(\FF^{\ell_\chi},\rho_\chi)$ is irreducible, i.e. (ii) holds. 

(iii)  By (i) we see that 
$\bigoplus_{\chi\in \widehat G/\Gamma}(\field^{\ell_\chi},\rho_\chi)
\cong \bigoplus_{\psi\in \widehat G}(\field,\psi)\cong (\mathcal{F}(G,\field),\rho_{\mathrm{reg}})$. 
So extending the base field from $\FF$ to $\field$ the representations 
$(\mathcal{F}(G,\FF),\rho_{\mathrm{reg}})$ and $\bigoplus_{\chi\in \widehat G/\Gamma}(\FF^{\ell_\chi},\rho_\chi)$ become isomorphic. 
It follows that these two representations are isomorphic over $\FF$. 

(iv) Statement (i) implies that for $\chi\neq \chi'$ in $\widehat G/\Gamma$, the representations 
$(\FF^{\ell_\chi},\rho_\chi)$ and $(\FF^{\ell_{\chi'}},\rho_{\chi'})$ are non-isomorphic, so the given list of irreducible representations is irredundant. 
On the other hand, 
every irreducible representation of $G$ over $\FF$ occurs as a direct summand 
in the regular representation, hence (iii) implies completeness of the list. 
 \end{proof} 

By Proposition~\ref{prop:irrep abelian} (iii) we have a direct sum decomposition 
\begin{equation}\label{eq:bigoplus V_chi}
\mathcal{F}(G,\FF)=\bigoplus_{\chi\in \widehat G/\Gamma}V_\chi 
\end{equation}  
into minimal $G$-invariant subspaces, where 
$(V_\chi,\rho_{\mathrm{reg}}\vert_{V_\chi})\cong (\FF^{\ell_\chi},\rho_\chi)$. 
Extending scalars from $\FF$ to $\field$ we get 
\begin{equation}\label{eq:bigoplus K tensor V_chi} 
\mathcal{F}(G,\field)=\bigoplus_{\chi\in \widehat G/\Gamma}\field\otimes_\FF V_\chi.
\end{equation} 

We need the following refinement of Proposition~\ref{prop:irrep abelian} (i): 

\begin{lemma}\label{lemma:f_psi}
$\field\otimes_\FF V_\chi$ has a $\field$-vector space basis 
$\{f_\psi\mid \psi\in \Gamma\cdot \chi\}$, where 
\begin{itemize}
\item[(i)] $\field f_\psi$ is a $G$-invariant subspace of $\field\otimes_\FF V_\chi$,  
on which the representation of $G$ is isomorphic to 
$(\field,\psi)$; 
\item[(ii)] 
$\{f_\psi\mid \psi\in \Gamma\cdot \chi\}$ is the $\Gamma$-orbit of 
$f_\chi$, so 
\begin{equation}\label{eq:gamma f} f_{\gamma\cdot \psi}=\gamma(f_\psi)\text{ for all }
\gamma\in \Gamma,\  
\psi\in \Gamma\cdot\chi.
\end{equation}  
\end{itemize} 
\end{lemma}
\begin{proof}
By Proposition~\ref{prop:irrep abelian} (i), $\field\otimes_\FF V_\chi$ contains a unique 
$1$-dimensional $G$-invariant subspace $\field v$ such that 
$(\field v,\rho_{\chi}\vert_{\field v})\cong (\field,\chi)$ (clearly $\field v$ is the $\omega_\chi$-eigenspace of 
$\rho_\chi(g_\chi)=A_\chi$). 
Denote by $\Gamma_0$ the stabilizer subgroup in $\Gamma$ of $\chi$, or equivalently, $\Gamma_0$ is the stabilizer in $\Gamma$ 
of $\omega_\chi$  
(so $\ell_\chi=|\Gamma\cdot \chi|=|\Gamma:\Gamma_0|$, the index of $\Gamma_0$ in $\Gamma$). 
Formula \eqref{eq:gamma v} shows that the stabilizer of $v$ in $\Gamma$ (and similarly, of any non-zero element of $\field v$) 
is contained in $\Gamma_0$. 
Pick an element $f_\chi$ in $\field v$ such that $f_\chi$ has a non-zero coordinate in the fixed field 
$\field^{\Gamma_0}$ of $\Gamma_0$ in $\field$. 
Then for any $\gamma\in \Gamma_0$, one non-zero coordinate of $f_\chi$ coincides with the corresponding coordinate of 
$\gamma(f_\chi)$. On the other hand,   \eqref{eq:gamma v} implies that for any $\gamma\in \Gamma_0$, 
$\gamma(f_\chi)$ is also an $\omega_\chi$-eigenvector of $A_\chi$, hence 
$\gamma(f_\chi)\in \field v=\field f_\chi$. It follows that $\gamma(f_\chi)=f_\chi$, so 
the stabilizer of $f_\chi$ in $\Gamma$ equals $\Gamma_0$, the stabilizer of $\chi$ in $\Gamma$. 
Now for $\psi\in \Gamma\cdot \chi$ we may define  $f_\psi$ as follows: take any $\gamma\in \Gamma$ with $\gamma\cdot \chi=\psi$, 
and set $f_\psi:=\gamma(f_\chi)$. By \eqref{eq:gamma v} we see that $f_\psi$ is a $\gamma(\omega_\chi)$-eigenvector of 
$\rho_\chi(g_\chi)$.  
Thus $\{f_\psi\mid \psi\in \Gamma\cdot \chi\}$ is a $\field$-vector space basis of $\field\otimes_\FF V_\chi$ 
satisfying (i) and (ii).  
\end{proof} 

Choose an $\FF$-vector space basis $\{e_\psi\mid \psi\in \Gamma\cdot \chi\}$ in 
$V_\chi$ (recall that $\dim(V_\chi)=\ell_\chi=|\Gamma\cdot \chi|$, so we can use 
the elements in the $\Gamma$-orbit of $\chi$ as labels for the basis vectors). 
Then by \eqref{eq:bigoplus V_chi}, 
$\{e_\psi\mid \psi\in \widehat G\}$ is an $\FF$-vector space basis of $\mathcal{F}(G,\FF)$, and 
also a $\field$-vector space basis of $\mathcal{F}(G,\field)$. 
By Lemma~\ref{lemma:f_psi} and \eqref{eq:bigoplus K tensor V_chi}, another $\field$-vector space basis of 
$\mathcal{F}(G,\field)$ is 
$\{f_\psi\mid \psi\in \widehat G\}$.  
The following statement identifies the $\FF$-subspace $\mathcal{F}(G,\FF)$ of 
$\mathcal{F}(G,\field)$ in terms of the coordinates with respect to the 
$\field$-vector space basis 
$\{f_\psi\mid \psi\in \widehat G\}$:  

\begin{lemma} \label{lemma:F(G,F)} 
We have 
\[\mathcal{F}(G,\FF)=
\{\sum_{\psi\in\widehat G} c_\psi f_\psi\in \mathcal{F}(G,\field)\mid c_{\gamma\psi}=\gamma(c_\psi)\in \field\text{ for all }
\gamma\in\Gamma, \ \psi\in\widehat  G\}.\] 
\end{lemma} 

\begin{proof} 
Recall that $\Gamma$ acts on $\mathcal{F}(G,\field)$ as $(\gamma(f))(g)=\gamma(f(g))$ for 
$\gamma\in \Gamma$, $f\in \mathcal{F}(G,\field)$ and $g\in G$. This is an action on the $\field$-vector space $\mathcal{F}(G,\field)$ 
via semilinear transformations. That is,  
for $f_1,f_2\in \mathcal{F}(G,\field)$, $c_1,c_2\in \field$, and $\gamma\in \Gamma$ we have 
$\gamma(c_1f_1+c_2f_2)=\gamma(c_1)\gamma(f_1)+\gamma(c_2)\gamma(f_2)$. 
Therefore we have 
\[\gamma(\sum_{\psi\in\widehat G} c_\psi f_\psi)=\sum_{\psi\in\widehat G} \gamma(c_\psi) \gamma(f_\psi)=
\sum_{\psi\in\widehat G} \gamma(c_\psi) f_{\gamma\cdot \psi}\] 
(the second equality above holds by \eqref{eq:gamma f} in Lemma~\ref{lemma:f_psi}), 
implying that 
\begin{equation}\label{eq:gamma(f)=f} 
\gamma(\sum_{\psi\in\widehat G} c_\psi f_\psi)=\sum_{\psi\in\widehat G} c_\psi f_\psi\iff 
\forall \psi\in\widehat G: \gamma(c_\psi)=c_{\gamma\cdot\psi}.
\end{equation}  
By definition of the $\Gamma$-action, and since $\FF$ is the fixed subfield of $\Gamma$ in $\field$, 
we have 
$\mathcal{F}(G,\FF)=\{f\in \mathcal{F}(G,\field)\mid \forall \gamma\in \Gamma:\gamma(f)=f\}$. 
So our statement follows by \eqref{eq:gamma(f)=f}.  
\end{proof} 

For a subset $M$ of $\mathcal{B}(\widehat G)\subset \mathcal{F}(G,\ZZ)$ and a subset 
$I\subseteq\widehat G$, set 
\[M_I:=\{m\in M\mid m(\psi)=0\text{ for all }\psi\in\widehat G\setminus I\},\]
and write $\langle M_I\rangle_\ZZ$ for the $\ZZ$-submodule of $\mathcal{F}(G,\ZZ)$ 
generated by $M_I$. 
Note that when $M$ is a submonoid of $\mathcal{B}(\widehat G)$, then 
$M_I$ is also a submonoid of $\mathcal{B}(\widehat G)$. 
The coordinate ring of $\mathcal{F}(G,\field)$ is a polynomial algebra 
$\field[x_\psi\mid \psi\in \widehat G]$, where $\{x_\psi \mid \psi\in \widehat G\}$ is a basis 
dual to the basis $\{f_\psi\mid \psi\in \widehat G\}$ in $\mathcal{F}(G,\field)$. 
In particular, we have that $g\cdot x_\psi=\psi(g)^{-1}x_\psi$ for $g\in G$. 
For $m\in \mathcal{B}(\widehat G)$ we set 
$x^m:=\prod_{\psi\in \widehat G}x_\psi^{m(\psi)}$. 
As we pointed out in \eqref{eq:B(G) span}, $\{x^m\mid m\in \mathcal{B}(\widehat G)\}$ is a $\field$-vector space basis 
of $\field[x_\psi\mid \psi\in\widehat G]^G$. 
If $m\in M_I$, then $x^m$ belongs to the subalgebra $\field[x_\psi\mid \psi\in I]$ of 
$\field[x_\psi\mid \psi\in\widehat G]$. 

\begin{theorem}\label{thm:separating monomials} 
Let $M$ be a subset of $\mathcal{B}(\widehat G)$. The polynomial 
$G$-invariants $\{x^m\mid m\in M\}$ separate the $G$-orbits in the subspace 
$\mathcal{F}(G,\FF)$ of $\mathcal{F}(G,\field)$ if  condition (*) below holds for $M$: 
\[(*)\text{ for all }\Gamma\text{-stable subsets}\ I \text{ of }\widehat G, \text{ we have }
\langle M_I\rangle_\ZZ\supseteq \mathcal{B}(\widehat G)_I.\] 
\end{theorem}

To prove Theorem~\ref{thm:separating monomials} we need Lemma~\ref{lemma:x^m(v)=0} below. 
Recall the direct sum decomposition 
$\mathcal{F}(G,\FF)=\bigoplus_{\chi\in \widehat G/\Gamma} V_\chi$ from \eqref{eq:bigoplus V_chi}, 
so $v\in \mathcal{F}(G,\FF)$ can be written as 
$v=(v_\chi\mid \chi\in \widehat G/\Gamma)$, where $v_\chi\in V_\chi$. 
Set $\mathrm{supp}(v):=\{\chi\in \widehat G/\Gamma\mid v_\chi\neq 0\}$.   

\begin{lemma}\label{lemma:x^m(v)=0} 
For $v\in \mathcal{F}(G,\FF)$, $\psi\in \widehat G$, and $m\in \mathcal{B}(\widehat G)$ we have 
\begin{itemize} 
\item[(i)] $x_\psi(v)\neq 0$ if and only if $\psi\in \bigcup_{\chi\in \mathrm{supp}(v)}\Gamma\cdot\chi$. 
\item[(ii)] $x^m(v)\neq 0$ if and only if $m\in \mathcal{B}(\widehat G)_I$ where $I= \bigcup_{\chi\in \mathrm{supp}(v)}\Gamma\cdot\chi$. 
\end{itemize}
\end{lemma}

\begin{proof} (i)  Recall that  $v_\chi=\sum_{\lambda\in \Gamma\cdot \chi}x_\lambda(v)f_\lambda$. 
If $x_\psi(v)$ is non-zero, then $v_\mu\neq 0$ for the unique $\mu\in \widehat G/\Gamma$ for which 
$\psi\in \Gamma\cdot \mu$. Then $\mu\in \mathrm{supp}(v)$, and so $\psi\in \bigcup_{\chi\in \mathrm{supp}(v)}\Gamma\cdot\chi$. 
For the reverse implication assume that $\psi\in \Gamma\cdot\chi$ for some $\chi\in \mathrm{supp}(v)$. So $v_\chi\neq 0$ 
and thus $x_\lambda(v)\neq 0$ for some $\lambda$ from $\Gamma\cdot\chi=\Gamma\cdot \psi$. 
Here $\psi,\lambda\in\widehat G$ belong to the same $\Gamma$-orbit in $\widehat G$, hence by Lemma~\ref{lemma:F(G,F)},   
$x_\lambda(v)\neq 0$ implies $x_\psi(v)\neq 0$. 

(ii) follows from (i). 
\end{proof}

\begin{proofof}{Theorem~\ref{thm:separating monomials}}
Suppose that $M\subseteq\mathcal{B}(\widehat G)$ satisfies condition (*), and 
$v,w\in \mathcal{F}(G,\FF)$ with $G\cdot v\neq G\cdot w$. 
We have to show that there exists an $m\in M$ with $x^m(v)\neq x^m(w)$. 

First we deal with the case when $\mathrm{supp}(v)\neq \mathrm{supp}(w)$. 
Then by symmetry we may assume that there exists a $\chi\in \mathrm{supp}(v)\setminus\mathrm{supp}(w)$. 
Let $I$ be the $\Gamma$-orbit of $\chi$. The element $\mathrm{ord}(\chi)\delta_{\chi}$ belongs to 
$\mathcal{B}(\widehat G)_I$, therefore by assumption (*) for $M$, there exist $m,m'\in M_I\subseteq \mathcal{B}(\widehat G)_I$ with 
$\mathrm{ord}(\chi)\delta_{\chi}=m-m'$. It follows that $m(\chi)>0$.  
By Lemma~\ref{lemma:x^m(v)=0} we have $x^m(v)\neq 0$, whereas $x^m(w)=0$. 
In particular, we found an $m\in M$ with $x^m(v)\neq x^m(w)$. 

From now on we assume that $\mathrm{supp}(v)=\mathrm{supp}(w)$. 
Since $G\cdot v\neq G\cdot w$, by \eqref{eq:separability by invariants} and \eqref{eq:B(G) span} we know that  
there exists a $b\in \mathcal{B}(\widehat G)$ with $x^b(v)\neq x^b(w)$. 
Set $I:= \bigcup_{\chi\in \mathrm{supp}(v)}\Gamma\cdot\chi= \bigcup_{\chi\in \mathrm{supp}(w)}\Gamma\cdot\chi$. 
Then $b\in \mathcal{B}(\widehat G)_I$ 
(otherwise by Lemma~\ref{lemma:x^m(v)=0}  (ii) 
we would have $x^b(v)=0=x^b(w)$, contrary to the choice of $b$). 
We can write $b$ as $b=m-m'$ for some $m,m'\in M_I$ by assumption (*) on $M$. 
By Lemma~\ref{lemma:x^m(v)=0} (ii) we know that none of $x^m(v)$, $x^{m'}(v)$, $x^m(w)$, $x^{m'}(w)$ is zero. 
Consequently, $x^b(v)=\frac{x^m(v)}{x^{m'}(v)}$ and  $x^b(w)=\frac{x^m(w)}{x^{m'}(w)}$. 
Therefore $x^b(v)\neq x^b(w)$ implies that $x^m(v)\neq x^m(w)$ or $x^{m'}(v)\neq x^{m'}(w)$, and we are done. 
\end{proofof} 

\begin{corollary} \label{cor:sepnoether}
Let $d$ be a positive integer satisfying that 
for all $\Gamma$-stable subsets $I$ of $\widehat G$, 
the monoid $\mathcal{B}(\widehat G)_I$ is contained in the $\ZZ$-submodule of 
$\mathcal{F}(G,\ZZ)$ generated by 
$\{m\in\mathcal{B}(\widehat G)_I\mid \sum_{\psi\in\widehat G}m(\psi)\le d\}$. 
Then we have 
\begin{itemize}
\item[(i)]  $\sepbeta(G,\mathcal{F}(G,\FF))\le d$. 
\item[(ii)] $\sepbeta^\FF(G)\le d$. 
\end{itemize} 
\end{corollary} 

\begin{proof} 
(i) Assume that the positive integer $d$ satisfies that for all $\Gamma$-stable subsets $I$ of $\widehat G$, 
the monoid $\mathcal{B}(\widehat G)_I$ is contained in the $\ZZ$-submodule of 
$\mathcal{F}(\widehat G,\ZZ)$ generated by 
$\{m\in\mathcal{B}(\widehat G)_I\mid \sum_{\psi\in\widehat G}m(\psi)\le d\}$. 
Setting $M:=\{m\in \mathcal{B}(\widehat G)\mid \sum_{\psi\in\widehat G}m(\psi)\le d\}$ we have 
that $M$ satisfies condition (*) from Theorem~\ref{thm:separating monomials}, 
hence the $G$-invariants 
$\{x^m\mid m\in M\}$ in $\field[\mathcal{F}(G,\field)]$ 
separate the $G$-orbits in the $\FF$-vector subspace 
$\mathcal{F}(G,\FF)$ of $\mathcal{F}(G,\field)$. 
Note that  $\mathrm{Span}_\field\{x^m\mid m\in M\}$ is 
the union of the homogeneous components of $\field[\mathcal{F}(G,\field)]^G$ of degree at most $d$ by 
\eqref{eq:B(G) span}. Recall that for all $j$, the degree $j$ homogeneous component of  $\field[\mathcal{F}(G,\field)]^G$ 
is spanned as a $\field$-vector space by the degree $j$ homogeneous component of $\FF[\mathcal{F}(G,\FF)]^G$. 
Since the homogeneous components of $\field[\mathcal{F}(G,\field)]^G$  of degree at most $d$ separate the $G$-orbits in 
$\mathcal{F}(G,\FF)$, the same holds for its spanning subset, consisting of the  elements of degree at most $d$ in  $\FF[\mathcal{F}(G,\FF)]^G$. 

(ii) Every finite dimensional representation of $G$  is a direct summand in a direct sum of several copies of the 
regular representation. Therefore it is sufficient to show that 
$\sepbeta^\FF(G,V)\le d$ where 
$V=\mathcal{F}(G,\FF)\oplus \cdots\oplus \mathcal{F}(G,\FF)=\mathcal{F}(G,\FF)^{\oplus n}$. 
Then $\field\otimes_\FF V\cong \bigoplus_{\psi\in\widehat G}(\field,\psi)^{\oplus n}$, 
and the coordinate ring $\field[\field\otimes_\FF V]$ has variables 
$\{x_\psi^{(1)},\dots,x_\psi^{(n)}\mid \psi\in\widehat G\}$ with 
$g\cdot  x_\psi^{(j)}=\psi^{-1}(g)x_\psi^{(j)}$. 
A variant of Lemma~\ref{lemma:F(G,F)}, Lemma~\ref{lemma:x^m(v)=0},  
and Theorem~\ref{thm:separating monomials} for this more general setup 
can be proved in the same way as before, the only difference is that 
the notation becomes considerably more complicated. 
Alternatively, if $\FF$ is large enough (say $\FF$ is infinite), 
then one may refer to \cite[Theorem 3.4 (ii)]{draisma-kemper-wehlau} implying that 
$\sepbeta(G,\mathcal{F}(G,\FF)^{\oplus n})=\sepbeta(G,\mathcal{F}(G,\FF))$. 
\end{proof} 

\begin{remark} 
A characterization of the subsets $M\subseteq \mathcal{B}(\widehat G)$ such that 
$\{x^m\mid m\in M\}$ separates the $G$-orbits in $\mathcal{F}(G,\field)$ was given in 
\cite[Theorem 2.1]{domokos:abelian} (see \cite[Proposition 3.2]{cahill-contreras-hip} for some additional information on 
separating monomials), yielding in \cite[Corollary 2.6]{domokos:abelian} a combinatorial characterization of $\sepbeta^\field(G)$ in terms of the block monoid of 
$\widehat G$.  Theorem~\ref{thm:separating monomials}  and Corollary~\ref{cor:sepnoether} generalize \cite[Proposition 2.3]{domokos:abelian} 
(one direction of \cite[Theorem 2.1]{domokos:abelian}) and one inequality of \cite[Corollary 2.6]{domokos:abelian} 
for the case of base fields that are not necessarily algebraically closed.  
\end{remark} 

As an application (and illustration) of Corollary~\ref{cor:sepnoether} we compute the 
separating Noether number over the field $\RR$ of real numbers of the 
direct product $\mathrm{C}_3\times\mathrm{C}_3$ of two copies of the cyclic group $\mathrm{C}_3$ 
of order $3$:  

\begin{proposition}\label{prop:C3xC3} 
We have the equality $\sepbeta^\RR(\mathrm{C}_3\times\mathrm{C}_3)=3$. 
\end{proposition} 

\begin{proof} 
Since $\mathrm{C}_3$ is a homomorphic image of $\mathrm{C}_3\times\mathrm{C}_3$ and $\sepbeta^\RR(\mathrm{C}_3)=3$ 
by  \cite[Proposition 3.13]{blumsmith-garcia-hidalgo-rodriguez}, we have $\sepbeta^\RR(\mathrm{C}_3\times\mathrm{C}_3)\ge 3$. 
The reverse inequality  $\sepbeta^\RR(\mathrm{C}_3\times\mathrm{C}_3)\le 3$ can be obtained applying 
Corollary~\ref{cor:sepnoether} for $\FF=\RR$, $\field=\CC$, and $G=\mathrm{C}_3\times\mathrm{C}_3$. 
Then $\Gamma$ is the automorphism group of the field extension $\CC$ of $\RR$, 
so $\Gamma$ is the two-element group generated by complex conjugation. 
In particular, the $\Gamma$-orbit of a character $\chi$ of $G$ consists of the character $\chi$ and 
its inverse $\chi^{-1}\in \widehat G$. The character group $\widehat G$ is also isomorphic to 
$\mathrm{C}_3\times\mathrm{C}_3$, so it is generated by two order $3$ characters $\psi_1$ 
and $\psi_2$. Writing $\chi_{(i,j)}:=\psi_1^i\psi_2^j$, we have 
$\widehat G=\{\chi_{(i,j)}\mid (i,j)\in \{0,1,2\}\times \{0,1,2\}$. 
The partition of $\widehat G$ into the disjoint union of $\Gamma$-orbits is the following: 
\[\widehat G=\{\chi_{(0,0)}\}\sqcup \{\chi_{(1,0)},\chi_{(2,0)}\}\sqcup
\{\chi_{(0,1)},\chi_{(0,2)}\}\sqcup\{\chi_{(1,1)},\chi_{(2,2)}\}\sqcup 
\{\chi_{(1,2)},\chi_{(2,1)}\}\]
In order to have a simpler notation for the elements of $\mathcal{F}(\widehat G,\ZZ)$, write 
$\delta_{(i,j)}:=\delta_{\chi_{(i,j)}}$ for the basis elements. 
The additive monoid $\mathcal{B}(\widehat G)$ is generated by its irreducible elements, 
where an element of  $\mathcal{B}(\widehat G)$ is said to be \emph{irreducible} 
if it is not the sum of two non-zero elements of $\mathcal{B}(\widehat G)$. 
Recall that the length of  $b\in \mathcal{B}(\widehat G)$ is $|b|=\sum_{\chi\in \widehat G}b(\chi)$. 
Clearly $\delta_{(0,0)}$ is the only length $1$ element of $\mathcal{B}(\widehat G)$, 
and $\delta_\chi+\delta_{\chi^{-1}}$ for $\chi\neq \chi_{(0,0)}$ are the only length $2$ irreducible elements  
in $\mathcal{B}_G$. The automorphism group of the group $\widehat G$ acts naturally by 
$\ZZ$-module automorphisms on $\mathcal{F}(\widehat G, \ZZ)$ and by monoid automorphisms of the 
submonoid $\mathcal{B}(\widehat G)$. We claim that up to such automorphisms, the only irreducible elements 
in $\mathcal{B}(\widehat G)$ with length greater than $2$ are 
\begin{align*} 3\delta_{(1,0)},\qquad 
\delta_{(1,0)}+\delta_{(0,1)}+\delta_{(2,2)}, \\
b_1:=\delta_{(1,0)}+\delta_{(0,1)}+2\delta_{(1,1)}, \qquad 
b_2:=\delta_{(1,0)}+2\delta_{(0,1)}+\delta_{(2,1)},\\ 
b_3:=2\delta_{(1,0)}+2\delta_{(0,1)}+\delta_{(1,1)}.
\end{align*} 
Indeed, if for an irreducible $b\in \mathcal{B}(\widehat G)$ with $|b|>2$ 
there are two different non-trivial characters $\sigma_1,\sigma_2$ with $b(\sigma_1)\neq 0$ and 
$b(\sigma_2)\neq 0$, then applying an automorphism of $\widehat G$ we may reduce to the case 
$\sigma_1=\chi_{(1,0)}$ and $\sigma_2=\chi_{(0,1)}$. 
Moreover, $b(\chi_{(1,0)})\in \{1,2\}$ and  $b(\chi_{(1,0)})\in \{1,2\}$.   
After these observations it is easy to sort out the few possibilities, and 
conclude the claim. In particular, $b_1,b_2,b_3$ represent the orbits of all irreducible elements in 
$\mathcal{B}(\widehat G)$ with length strictly greater than $3$. 

Set $M:=\{b\in \mathcal{B}(\widehat G)\mid |b|\le 3\}$  
If  $b_1\in \mathcal{B}(\widehat G)_I$ for some $\Gamma$-stable subset $I$ 
of $\widehat G$, then  
$\chi_{(1,0)},\chi_{(0,1)},\chi_{(1,1)},\chi_{(2,2)}\in I$, and 
the equality 
\[\delta_{(1,0)}+\delta_{(0,1)}+2\delta_{(1,1)}=
(\delta_{(1,0)}+\delta_{(0,1)}+\delta_{(2,2)})+2(\delta_{(1,1)}+\delta_{(2,2)})-3\delta_{(2,2)}\] 
shows that $b_1\in \langle M_I\rangle_\ZZ$. 

If  $b_2\in \mathcal{B}(\widehat G)_I$ for some $\Gamma$-stable subset $I$ 
of $\widehat G$, then  
$\chi_{(1,0)},\chi_{(0,1)},\chi_{(2,1)},\chi_{(0,2)}\in I$, and 
the equality 
\[\delta_{(1,0)}+2\delta_{(0,1)}+\delta_{(2,1)}
=(\delta_{(1,0)}+\delta_{(0,2)}+\delta_{(2,1)})+2(\delta_{(0,1)}+\delta_{(0,2)})-3\delta_{(0,2)}\] 
shows that $b_2\in \langle M_I\rangle_\ZZ$. 

If  $b_3\in \mathcal{B}(\widehat G)_I$ for some $\Gamma$-stable subset $I$ 
of $\widehat G$, then 
$\chi_{(1,0)},\chi_{(0,1)},\chi_{(1,1)},\chi_{(2,2)}\in I$, and  
the equality 
\[2\delta_{(1,0)}+2\delta_{(0,1)}+\delta_{(1,1)}=2(\delta_{(1,0)}+\delta_{(0,1)}+\delta_{(2,2)})
+(\delta_{(1,1)}+\delta_{(2,2)})-3\delta_{(2,2)}\] 
shows that $b_3\in \langle M_I\rangle_\ZZ$. 

Thus we verified that the condition of Corollary~\ref{cor:sepnoether} is satisfied by $d=3$, 
hence the desired inequality $\sepbeta^\RR(\mathrm{C}_3\times\mathrm{C}_3)\le 3$ 
holds. 
\end{proof}

\begin{remark} 
It is shown in \cite[Theorem 1.2]{schefler_c_n^r} that $\sepbeta^\CC(\mathrm{C}_3\times\mathrm{C}_3)=4$ 
(the inequality $\sepbeta^\CC(\mathrm{C}_3\times\mathrm{C}_3)\le4$ follows essentially from the equality 
$b_3=b_1+b_1-3\delta_{(1,1)}$). 
Moreover, $\beta^\RR(\mathrm{C}_3\times\mathrm{C}_3)=\beta^\CC(\mathrm{C}_3\times\mathrm{C}_3)=5$ 
(see \cite{cziszter-domokos-szollosi} for information on $\beta^\FF(G)$).  
\end{remark}

\section{An example} \label{sec:an example}

We give an example showing that the results of the present paper yield a notable 
decrease in the size of a separating set of invariants for a real representation 
of a finite abelian group, compared to systems of polynomial invariants whose separating property is guaranteed 
by the prior known results.  
Consider $V:=\RR^6$ endowed with the natural representation of the matrix group $G$ generated by 
\[g_1:=\left(\begin{array}{rrrrrr}
-\frac{1}{2} & -\frac{\sqrt{3}}2& 0 & 0 & 0 & 0 \\
\frac{\sqrt{3}}2& -\frac{1}{2} & 0 & 0 & 0 & 0 \\
0 & 0 & 1 & 0 & 0 & 0 \\
0 & 0 & 0 & 1 & 0 & 0 \\
0 & 0 & 0 & 0 & -\frac{1}{2} & -\frac{\sqrt{3}}2\\
0 & 0 & 0 & 0 & \frac{\sqrt{3}}2& -\frac{1}{2}
\end{array}\right), 
\qquad g_2:=\left(\begin{array}{rrrrrr}
1 & 0 & 0 & 0 & 0 & 0 \\
0 & 1 & 0 & 0 & 0 & 0 \\
0 & 0 & -\frac{1}{2} & -\frac{\sqrt{3}}2& 0 & 0 \\
0 & 0 & \frac{\sqrt{3}}2& -\frac{1}{2} & 0 & 0 \\
0 & 0 & 0 & 0 & -\frac{1}{2} & -\frac{\sqrt{3}}2\\
0 & 0 & 0 & 0 & \frac{\sqrt{3}}2& -\frac{1}{2}
\end{array}\right).\]
The group $G$ is isomorphic to $\mathrm{C}_3\times\mathrm{C}_3$, hence by Proposition~\ref{prop:C3xC3}, 
$\sepbeta(G,V)\le 3$. 
A minimal homogeneous system of generators of $\RR[V]^G$ can be computed by the online CoCalc platform \cite{CoCalc}, and 
we get that the $G$-orbits in $V$ are separated by the following invariants 
(the elements of degree at most $3$ in the minimal homogeneous generating system):   
\begin{align*}x_{5}^{2} + x_{6}^{2}, \quad x_{3}^{2} + x_{4}^{2}, \quad x_{1}^{2} + x_{2}^{2}, 
\quad x_{5}^{2} x_{6} - \frac{1}{3} x_{6}^{3},\quad  x_{5}^{3} - 3 x_{5} x_{6}^{2}, 
\\ \quad x_{2} x_{3} x_{5} + x_{1} x_{4} x_{5} - x_{1} x_{3} x_{6} + x_{2} x_{4} x_{6}, 
\quad x_{1} x_{3} x_{5} - x_{2} x_{4} x_{5} + x_{2} x_{3} x_{6} + x_{1} x_{4} x_{6}, 
\\ x_{3}^{2} x_{4} - \frac{1}{3} x_{4}^{3}, \quad x_{3}^{3} - 3 x_{3} x_{4}^{2}, \quad x_{1}^{2} x_{2} - \frac{1}{3} x_{2}^{3},\quad  x_{1}^{3} - 3 x_{1} x_{2}^{2}
\end{align*} 
The known results prior to this paper guaranteed that 
$\sepbeta^\RR(G,V)\le\sepbeta^\CC(G,\CC\otimes_\RR V)\le 4$ (see \cite{domokos:abelian}). 
Diagonalizing the representation over $\CC$, one can deduce from \cite[Proposition 2.4]{domokos:abelian} that 
in order to separate the orbits in the complexification $\CC[\CC\otimes_\RR V]$, 
we need to include the following $6$ degree $4$ invariants: 
\begin{align*} 
x_{1} x_{3} x_{5}^{2} - x_{2} x_{4} x_{5}^{2} - 2 x_{2} x_{3} x_{5} x_{6} - 2 x_{1} x_{4} x_{5} x_{6} - x_{1} x_{3} x_{6}^{2} + x_{2} x_{4} x_{6}^{2}, 
\\ x_{2} x_{3} x_{5}^{2} + x_{1} x_{4} x_{5}^{2} + 2 x_{1} x_{3} x_{5} x_{6} - 2 x_{2} x_{4} x_{5} x_{6} - x_{2} x_{3} x_{6}^{2} - x_{1} x_{4} x_{6}^{2}, 
\\ x_{1} x_{3}^{2} x_{5} + 2 x_{2} x_{3} x_{4} x_{5} - x_{1} x_{4}^{2} x_{5} + x_{2} x_{3}^{2} x_{6} - 2 x_{1} x_{3} x_{4} x_{6} - x_{2} x_{4}^{2} x_{6}, 
\\ x_{2} x_{3}^{2} x_{5} - 2 x_{1} x_{3} x_{4} x_{5} - x_{2} x_{4}^{2} x_{5} - x_{1} x_{3}^{2} x_{6} - 2 x_{2} x_{3} x_{4} x_{6} + x_{1} x_{4}^{2} x_{6}, 
\\ x_{1}^{2} x_{3} x_{5} - x_{2}^{2} x_{3} x_{5} + 2 x_{1} x_{2} x_{4} x_{5} - 2 x_{1} x_{2} x_{3} x_{6} + x_{1}^{2} x_{4} x_{6} - x_{2}^{2} x_{4} x_{6}, 
\\x_{1} x_{2} x_{3} x_{5} - \frac{1}{2} x_{1}^{2} x_{4} x_{5} + \frac{1}{2} x_{2}^{2} x_{4} x_{5} + \frac{1}{2} x_{1}^{2} x_{3} x_{6} - \frac{1}{2} x_{2}^{2} x_{3} x_{6} + x_{1} x_{2} x_{4} x_{6}\end{align*}   
The minimal generating system of $\RR[V]^G$ is even larger, since it contains also the following $6$ degree $5$ invariants: 
\begin{align*} 
x_{2} x_{3}^{2} x_{5}^{2} - 2 x_{1} x_{3} x_{4} x_{5}^{2} - x_{2} x_{4}^{2} x_{5}^{2} + 2 x_{1} x_{3}^{2} x_{5} x_{6} 
+ 4 x_{2} x_{3} x_{4} x_{5} x_{6} - 2 x_{1} x_{4}^{2} x_{5} x_{6} 
\\ - x_{2} x_{3}^{2} x_{6}^{2} + 2 x_{1} x_{3} x_{4} x_{6}^{2} + x_{2} x_{4}^{2} x_{6}^{2}, 
\\ x_{1} x_{3}^{2} x_{5}^{2} + 2 x_{2} x_{3} x_{4} x_{5}^{2} - x_{1} x_{4}^{2} x_{5}^{2} - 2 x_{2} x_{3}^{2} x_{5} x_{6} + 4 x_{1} x_{3} x_{4} x_{5} x_{6} 
+ 2 x_{2} x_{4}^{2} x_{5} x_{6} 
\\ - x_{1} x_{3}^{2} x_{6}^{2} - 2 x_{2} x_{3} x_{4} x_{6}^{2} + x_{1} x_{4}^{2} x_{6}^{2}, 
\\ x_{1} x_{2} x_{3} x_{5}^{2} - \frac{1}{2} x_{1}^{2} x_{4} x_{5}^{2} + \frac{1}{2} x_{2}^{2} x_{4} x_{5}^{2} - x_{1}^{2} x_{3} x_{5} x_{6} + x_{2}^{2} x_{3} x_{5} x_{6} 
- 2 x_{1} x_{2} x_{4} x_{5} x_{6} 
\\ - x_{1} x_{2} x_{3} x_{6}^{2} + \frac{1}{2} x_{1}^{2} x_{4} x_{6}^{2} - \frac{1}{2} x_{2}^{2} x_{4} x_{6}^{2}, 
\\ x_{1}^{2} x_{3} x_{5}^{2} - x_{2}^{2} x_{3} x_{5}^{2} + 2 x_{1} x_{2} x_{4} x_{5}^{2} + 4 x_{1} x_{2} x_{3} x_{5} x_{6} - 2 x_{1}^{2} x_{4} x_{5} x_{6} 
+ 2 x_{2}^{2} x_{4} x_{5} x_{6} 
\\ - x_{1}^{2} x_{3} x_{6}^{2} + x_{2}^{2} x_{3} x_{6}^{2} - 2 x_{1} x_{2} x_{4} x_{6}^{2}, 
\\ x_{1} x_{2} x_{3}^{2} x_{5} + x_{1}^{2} x_{3} x_{4} x_{5} - x_{2}^{2} x_{3} x_{4} x_{5} - x_{1} x_{2} x_{4}^{2} x_{5} + \frac{1}{2} x_{1}^{2} x_{3}^{2} x_{6} 
- \frac{1}{2} x_{2}^{2} x_{3}^{2} x_{6}
\\ - 2 x_{1} x_{2} x_{3} x_{4} x_{6} - \frac{1}{2} x_{1}^{2} x_{4}^{2} x_{6} + \frac{1}{2} x_{2}^{2} x_{4}^{2} x_{6}, 
\\ x_{1}^{2} x_{3}^{2} x_{5} - x_{2}^{2} x_{3}^{2} x_{5} - 4 x_{1} x_{2} x_{3} x_{4} x_{5} - x_{1}^{2} x_{4}^{2} x_{5} + x_{2}^{2} x_{4}^{2} x_{5} 
- 2 x_{1} x_{2} x_{3}^{2} x_{6} 
\\ - 2 x_{1}^{2} x_{3} x_{4} x_{6} + 2 x_{2}^{2} x_{3} x_{4} x_{6} + 2 x_{1} x_{2} x_{4}^{2} x_{6}
\end{align*}

\section*{Acknowledgements} 

This research was partially supported by the Hungarian National Research, Development and Innovation Office,  NKFIH K 138828.


\end{document}